\documentclass[reqno]{amsart}
\usepackage{amsmath}
\usepackage{dsfont}
\usepackage{amsfonts}
\usepackage{amssymb}
\usepackage{aliascnt}
\usepackage{graphicx}
\usepackage{mathrsfs}
\usepackage{enumerate}
\usepackage[bookmarks=true,pdfstartview=FitH, pdfborder={0 0 0}, colorlinks=true,citecolor=blue, linkcolor=blue]{hyperref}
\usepackage{tikz}
\newtheorem{theorem}{Theorem}[section]
\newaliascnt{conj}{theorem}
\newaliascnt{cor}{theorem}
\newaliascnt{lemma}{theorem}
\newaliascnt{fact}{theorem}
\newaliascnt{claim}{theorem}
\newaliascnt{prop}{theorem}
\newaliascnt{definition}{theorem}

\newtheorem{cor}[cor]{Corollary}
\newtheorem{lemma}[lemma]{Lemma}

\newtheorem{prop}[prop]{Proposition}

\aliascntresetthe{conj} \aliascntresetthe{cor}
\aliascntresetthe{lemma} \aliascntresetthe{fact}
\aliascntresetthe{claim} \aliascntresetthe{prop}
\aliascntresetthe{definition}

\theoremstyle{definition}
\newaliascnt{example}{theorem}

\aliascntresetthe{example}

\theoremstyle{remark}
\newaliascnt{rmk}{theorem}

\aliascntresetthe{rmk} 

\def\sek~{\S{}}

\numberwithin{equation}{section}

\newcommand{\diam}{\operatorname{diam}}
\newcommand{\Cone}{\operatorname{Cone}^{1}}

\newcommand{\dis}{\operatorname{d}}

\newcommand{\rad}{\operatorname{Rad}}

\newcommand{\alex}{\operatorname{Alex}}

\newcommand{\ZZ}{\mathds{Z}}

\renewcommand{\SS}{\mathbf{S}}

\newcommand{\RP}{\mathds{RP}}
\newcommand{\CP}{\mathds{CP}}
\newcommand{\HP}{\mathds{HP}}
\newcommand{\CaP}{\rm{Ca}\mathds{P}}

\begin{document}
\title{Rigidity for Positively Curved Alexandrov Spaces with Boundary}
\author{Jian Ge}
\address[Ge]{Beijing International center for Mathematical Research, Peking University. Beijing 100871, China}
\email{jge@math.pku.edu.cn}

\author{Ronggang Li}
\address[Li]{School of Mathematical Science, Peking University. Beijing 100871, China}
\email{lrg@pku.edu.cn}


\begin{abstract}
Inspired by a recent work of Grove-Petersen in \cite{GP2018}, where the authors studied Alexandrov spaces with largest possible boundary. We study  Alexandrov spaces with lower curvature bound $1$ and with small boundary. When the radius of $X$ is $\pi/2$,  and the boundary has diameter $\pi/2$, we classify the total space $X$. 
\end{abstract}
%-------------------------------------------------- Section 0------------------------------------------------
\maketitle
\section{Introduction}
Recall the radius of a metric space $(X, \dis)$ at $p$ is defined by
$$
\rad_{p}(X)=\sup\{\dis(p, x)| x\in X\},
$$
and the radius of $X$ is defined by
$$
\rad(X)=\inf_{p\in X}\rad_{p}(X).
$$
The celebrated Radius Sphere Theorem of Grove-Petersen \cite{GP1993} says that if $X$ is an $n$-dimensional Alexandrov space with lower curvature bound $1$ and $\rad(X)>\pi/2$, then $X$ is homeomorphic to a sphere. The $\pi/2$ is sharp in the sense that the projective spaces have radius $\pi/2$. On the other hand all positively curved Alexandrov spaces with lower curvature bound $1$ can be rescaled to ones with arbitrary small radius, it is interesting to study those spaces with radius equals to $\pi/2$. In this note we will use $\alex^{n}(\kappa)$ to denote the set of all complete $n$-dimensional Alexandrov spaces with lower curvature bound $\kappa$ and use $\mathcal{M}^{n}(\kappa)$ to denote the set of all complete smooth $n$-dimensional Riemannian manifolds with lower sectional curvature bound $\kappa$. We are interested in the following set:
$$
\mathcal{A}:=\{X\in \alex^{n}(1)| \rad(X)=\pi/2\}.
$$

For the Riemannian case, the diameter rigidity theorem of Gromoll and Grove (cf. \cite{GG1987} and \cite{Wil2001}) says that if $M$ is a smooth Riemannian manifold with sectional curvature lower bound $1$ and diameter $\pi/2$, then $M$ is either homeomorphic to the sphere or \emph{isometric} to one of the Compact Rank One Symmetric Spaces (CROSS for short). Moreover, if $\rad(M)=\pi/2$, then $M$ is isometric to one of the projective spaces, i.e. $\RP^{n}, \CP^{n}, \HP^{n}, \CaP^{2}$ with the canonical metric normalized such that sectional curvatures lie in $[1, 4]$. 

In the Alexandrov setting, one verifies easily that for any $X, Y\in \mathcal{A}$, the spherical join $X*Y\in \mathcal{A}$. Other examples can be constructed via group actions, see discussion of rigidity part in \cite{GP1993}. Therefore, $\mathcal{A}$ contains many more spaces than the Riemannian setting. It is a challenging problem to classify all such spaces. cf \cite{RW2016} for a recent rigidity result.

Inspired by a recent result of Grove-Petersen \cite{GP2018}, we study the following subset of $\mathcal{A}$:
$$
\mathcal{B}:=\{X\in \alex^{n}(1)| \rad(X)=\pi/2, \partial X\in \mathcal{M}^{n-1}(1)\}.
$$
When we write $\partial X\in \mathcal{M}$, we consider $\partial X$ as an inner metric space in the sense that the distance between two points in $\partial X$ is the infimum of all paths connecting them in $\partial X$. Since the curvature of $X$ is positive, by \cite{Per1991}, the distance function $\rho(\cdot):=\dis(\partial X, \cdot)$ is strictly concave in $X$. Therefore there exists a unique point, denoted by $s\in X$ at which $\rho$ achieves the maximum. $s\in X$ will be called the point soul of $X$. Consequently, $X$ is homeomorphic to the cone over its boundary $\partial X$. Therefore, it is only interesting to study the \emph{geometry} of such spaces. 

To state our theorem, we need the following notations. For $\theta\in (0, \pi]$, the \emph{Alexandrov Lens} is defined as the spherical join (cf. \cite{GP2018})
$$
L^{n}_{\theta}=\SS^{n-2}*[0, \theta],
$$
where $\SS^{n-2}$ is the unit $(n-2)$-dimensional sphere. The space $L_{\theta}^{n}$ admits a canonical isometric $\ZZ_{2}:=\ZZ/2\ZZ$ action, induced by the antipodal action on the factor $\SS^{n-2}$ and reflection with respect to $\theta/2$ on the other factor $[0, \theta]$. The quotient space $L^{n}_{\theta}/\ZZ_{2}$ is homeomorphic to the cone over $\RP^{n-1}$, with boundary isometric to the round $\RP^{n-1}$ of constant curvature $1$. For any $Y\in \alex^{n-1}(1)$, we denote the spherical cone $[0, \pi/2]\times_{\sin(t)}Y$ by $\Cone(Y)$.

In order to study $\mathcal B$, we study the following closed related class:
$$
\mathcal{B}':=\{X\in \alex^{n}(1)| \rad_{s}(X)=\pi/2, \partial X\in \mathcal{M}^{n-1}(1)\}.
$$
We have the following rigidity theorem of $\mathcal{B}'$:
\begin{theorem}\label{thm:main}
Let $X^{n}\in \alex^{n}(1)$ with $\partial X\in \mathcal{M}^{n-1}(1)$ and $\rad_{s}(X)=\pi/2$, then
\begin{enumerate}
\item (\cite{GP2018}) If $\partial X=\SS^{n-1}$, then $X=L_{\theta}^{n}$, for $\theta\in (0, \pi]$;
\item If $\diam(\partial X)< \pi/2$, then $X=\Cone(\partial X)$,
\item If $\diam(\partial X)=\pi/2$ and $\partial X$ is not homeomorphic to $\SS^{n-1}$, then the following hold:
\begin{enumerate}
\item If $\partial X=\RP^{n-1}$, then $X=L^{n}_{\theta}/\ZZ_{2}$, for $\theta\in (0, \pi]$.
\item If $\partial X=\CP^{m}, \HP^{m}, \CaP^{2}$, then $X=\Cone(\partial X)$.
\end{enumerate}
\end{enumerate}
\end{theorem}

Note that the case (1) in the theorem above is proved by Grove-Petersen in \cite{GP2018}. There is one case, i.e. $\rad(\partial X)\ge\pi/2$ except CROSSes, when we do not have rigidity result. \autoref{thm:main} implies the structures of $\mathcal{B}$. If $\partial X$ is a spherical space form, we can drop the radius assumption, that is
\begin{theorem}\label{thm:spaceform}
Let $X\in \alex^{n}(1)$ with $\partial X$ isometric to an $(n-1)$-dimensional spherical space form. If $\partial X=\RP^{n-1}$ then $X=L^{n}_{\theta}/\ZZ_{2}$, for $\theta\in (0, \pi]$. Otherwise $X=\Cone(\partial X)$.
\end{theorem}

Acknowledgment: It is our great pleasure to thank Fernando Galaz-Garc\'ia for comments and careful reading of a preliminarily version of the paper. We also thank Luis Guijarro for his interest in our work.

%----------------------------------------------------------------------------------------------------------------------------------------------
\section{Proof of the Main Theorem}
We study the subset that consists of points with maximal distance to the soul $s$ first. Let
$$
E:=\{x\in \partial X| \dis(x, s)=\pi/2\}.
$$
The set $E$ is nonempty under the hypotheses of \autoref{thm:main}. The set of points on the boundary that are foot points of the soul $s$ is denoted by:
$$
A:=\{x\in \partial X| \dis(s, x)=\dis(s, \partial X)\}.
$$
The pair of sets $(E, A)$ is called dual pair of $\partial X$, which plays an important role in the study of positively curved spaces. They either coincide or are $\pi/2$-part. When they coincide, we have the following consequence of the rigidity part from Toponogov Comparison Theorem.
\begin{lemma}\label{lem:cone}
If $E=A$, then $X=\Cone(\partial X)$.
\end{lemma}

When $E\ne A$, we want to show they are far apart. Before proving that, let's recall: $Y\subset X$ is called $\ell$-convex if for any geodesic in $X$ with length $<\ell$ and both ends lie in $Y$ remains entirely in $Y$. Throughout this note, we call $\gamma$ a geodesic in $\partial X$, if it is locally geodesic with respect to the induced inner metric from $X$. Therefore, it is a quasigeodesic of the ambient space $X$ by the Liberman's theorem proved in \cite{PP1993}. We use $|x,y|_{X}$ or simply $|x, y|$ as shorthand for the distance between $x$ and $y$ in the metric space $X$.
%-------------------------------------------------------|EA|\ge \pi/2-----------------------------------------------
\begin{prop}\label{lem:key}
For $x\in E$ and $p\in \partial X \setminus E$. Let $\gamma:[0,\ell]\rightarrow \partial X$ be a unit speed geodesic lying in $\partial X$, connecting $p=\gamma(0)$ and $x=\gamma(\ell)$. If there exists $\uparrow_{p}^{s}\in \Uparrow_{p}^{s}$ such that
\begin{equation}\label{eq:lem:key:01}
|\uparrow_{p}^{s}, \dot{\gamma}^{+}(0)|=\frac{\pi}{2},
\end{equation}
then, we have
$$
|xp|_{\partial X}\geq\frac{\pi}{2}  {\rm\ \ and\ \ }|pE|_{\partial X}\geq \frac{\pi}{2}.
$$
\end{prop}  
\begin{proof}
Applying the hinge version of Toponogov comparison theorem for the hinge based at $p$, whose sides are $\gamma$ and a geodesic from $p$ to $s$ corresponding to $\uparrow_{p}^{s}$, we get:
$$
\cos|sx|\geq \cos|sp|\cos|xp|_{\partial X}+\sin|sp|\sin(|xp|_{\partial X})\cos|\uparrow_{p}^{s},\dot{\gamma}^{+}(0)| .
$$
Using the condition \eqref{eq:lem:key:01}, we have
$$ 
0\geq\cos|sp|\cos(\ell)
$$
Since $p\notin E$, $|ps|<\pi/2$. Therefore
$$
|xp|_{\partial X}=\ell\ge \frac{\pi}{2}.
$$
It follows that $|pE|_{\partial X}\geq\frac{\pi}{2}$.
\end{proof}
%------------------------------------------------------------------------------------------------------------------------------------------------
Condition \eqref{eq:lem:key:01} is fulfilled in many cases. For example if $E\ne A$ and thus we can find a footpoint $p\in A$ which is not in $E$. The first variational formula applied to the distance function $\rho=\dis_{\partial X}(\cdot)$ yields $|\uparrow_{p}^{s},\partial\Sigma_{p}X|=\frac{\pi}{2}$. Therefore \eqref{eq:lem:key:01} holds. We just proved: 
\begin{cor}\label{cor:diam}
If $E\ne A$. Let $p\in \partial X$ be a footpoint of the soul point $s$. We have
$$
|pE|_{\partial X}\geq \frac{\pi}{2}
$$
\end{cor}

\begin{cor}[Case 2 of \autoref{thm:main}]\label{cor:small}
Let $X\in \alex^{n}(1)$ with non-empty boundary $\partial X$. If $\diam \partial X<\frac{\pi}{2}$ then $X$ is isometric to the spherical cone over $\partial X$
$$
X=\Cone(\partial X).
$$
\end{cor}
\begin{proof}
If $E\ne A$, then by \autoref{cor:diam}, there exists $x\in E$ and $p\in A$ such that $|px|_{\partial X}\ge \pi/2$. Hence we get a contradiction to the assumption $\diam(\partial X)<\pi/2$. Therefore $E=A$. By \autoref{lem:cone}, the conclusion holds.
\end{proof}

Next, we study the rigidity case when the lower bound of $\ell$ in \autoref{lem:key} is achieved. Namely, $\ell=\pi/2$.
%------------------------------------------------------------------------------------------------------------------------------------------------
\begin{prop}\label{prop:key:rigidity}
Suppose $\gamma:[0, {\pi}/{2}]\rightarrow \partial X$ is a geodesic in $\partial X$ with $\gamma(0)=p$, satisfying $|sp|<\frac{\pi}{2}$ and $x=\gamma({\pi}/{2})\in E$ with$|\uparrow_p^s,\dot{\gamma}^{+}(0)|=\frac{\pi}{2}.$ Then, 
$$|\uparrow^s_x,\dot{\gamma}^{-}(\frac{\pi}{2})|=|sp|.$$
\end{prop}	
\begin{proof}
Using the hinge version of Toponogov comparison theorem for the hinge based at $x\in E$, whose sides are geodesic connecting $x$ and $s$ and $\gamma[0,{\pi}/{2}]$, we have:
$$
\cos|sp|\geq \cos|xs|\cos|xp|_{\partial X}+\sin|xs|\sin|xp|_{\partial X}\cos|\uparrow_{x}^{s},\dot{\gamma}^{-}(\frac{\pi}{2})|
$$
Using the fact $|xs|=|xp|_{\partial X}={\pi}/{2}$, we have
$$ 
\cos|sp|\geq \cos|\uparrow_{x}^{s},\dot{\gamma}^{-}(\frac{\pi}{2})|,
$$
which implies
\begin{equation}\label{eq:prop:key:rigidity:01}
|sp|\le|\uparrow_{x}^{s},\dot{\gamma}^{-}(\frac{\pi}{2})|.
\end{equation}
On the other hand, for any small $\epsilon>0$, using the hinge version of Toponogov comparison theorem for the hinge based at $p$, whose sides are geodesics connecting $p$ and $s$ and $\gamma[0,{\pi}/{2}-t]$ for $t\in [0,\epsilon]$, we have
$$
\cos|s\gamma(\frac{\pi}{2}-t)|\geq \cos|sp|\cos(\frac{\pi}{2}-t)+\sin|sp|\sin(\frac{\pi}{2}-t)\cos|\uparrow_{p}^{s},\dot{\gamma}^{+}(0)| 
$$
Using the fact $|\uparrow_{p}^{s},\dot{\gamma}^{+}(0)|={\pi}/{2}$, we get
$$ 
\cos|s\gamma(\frac{\pi}{2}-t)|\geq  \cos|sp|\cos(\frac{\pi}{2}-t)
$$
that is
$$
\frac{\cos|s\gamma(\frac{\pi}{2}-t)|}{\cos(\frac{\pi}{2}-t)}\geq\cos|sp|.
$$
Therefore by the first variational formulae as $t\to 0$, we have
$$
\cos|sp|\leq\cos|\uparrow_{x}^{s},\dot{\gamma}^{-}(\frac{\pi}{2})|.
$$
that is:
\begin{equation}\label{eq:prop:key:rigidity:02}
|sp|\geq|\uparrow_{x}^{s},\dot{\gamma}^{-}(\frac{\pi}{2})| .
\end{equation}
Combinning \eqref{eq:prop:key:rigidity:01} and \eqref{eq:prop:key:rigidity:02}, we get
$$
|sp|=|\uparrow_{x}^{s},\dot{\gamma}^{-}(\frac{\pi}{2})|.
$$
\end{proof}

\begin{lemma}[Lemma 3.1 in \cite{GP2018}]\label{lem:E_convex}
$E$ is $\pi$-convex in $\partial X$. For any geodesic $\gamma:[0,l]\rightarrow \partial X$ that entirely lies in $E$, we have:
$$
|\uparrow_{\gamma(t)}^s,\dot{\gamma}^{\pm}(t)|=\frac{\pi}{2} \qquad \forall t\in[0,l]
$$ 
\end{lemma}

Hence by \autoref{lem:E_convex}, $E\subset \partial X$ is a convex set and in particular an Alexandrov space. We want to study the dimension of $E$. The dimension puts a strong restriction on the boundary topology as we will see.

\begin{theorem}\label{thm:key:dimension}
Let $X\in\alex^{n} (1)$ with boundary $\partial X$. If $\partial X$ is isometric to a compact rank one symmetric space (CROSS). Then $E$ is a closed submanifold of $\partial X$ without boundary with dimension $\ge (n-2)$.
\end{theorem}
\begin{proof} 
If $\partial X=\SS^{n-1}$, then the claim follows from \cite{GP2018}. Suppose $\partial X$ is not a sphere, then we know all the geodesics are closed in $\partial X$ with a common period $\pi$. For any $x\in E$, let $\gamma:[0,\pi]\rightarrow\partial X$ be a closed geodesic with $\gamma(0)=\gamma(\pi)=x$ in $\partial X$. Therefore, $\gamma$ is a quasigeodesic in $X$. Suppose that $\gamma$ does not lie entirely in $E$. Then 

{\bf Claim1:} $t_{0}=\pi/2$ is the unique minimum point for the function $\dis(s, \gamma(t))$ for $t\in [0, \pi]$, i.e.
\begin{equation}\label{eq:claim1:-1}
|\gamma(t_{0}),s|<|\gamma(t),s| \qquad \forall t\in[0, \pi/2)\cup (\pi/2, \pi],
\end{equation}
Moreover
\begin{equation}\label{eq:claim1:0}
|\uparrow^{s}_{\gamma(\frac{\pi}{2})}, \dot{\gamma}^{+}(\frac{\pi}{2})|=|\uparrow^{s}_{\gamma(\frac{\pi}{2})}, \dot{\gamma}^{-}(\frac{\pi}{2})|=\frac{\pi}{2}.
\end{equation}
\begin{proof}[Proof of Claim 1]
Since $\gamma$ does not lie entirely in $E$, there exists $0<t_{0}<\pi$, such that $\gamma(t_{0})$ is closest to $s$ along $\gamma$. By the first variational formula, we have
$$|\dot{\gamma}^{\pm}(t_{0}), \uparrow_{\gamma(t_{0})}^{s}|=\pi/2.$$ 
Therefore one can apply \autoref{lem:key} to the triangle formed by the quasigeodesic $\gamma[0, t_{0}]$ and geodesics $ps$ and $sx$. Therefore we have $t_{0}\ge \pi/2$. Apply \autoref{lem:key} to the triangle formed by the quasigeodesic $\gamma[t_{0}, \pi]$ and geodesics $ps$ and $sx$, one get $\pi-t_{0}\ge \pi/2$. Therefore $t_{0}=\pi/2$.
\end{proof}

Therefore by \autoref{prop:key:rigidity}, we have:
\begin{equation}\label{eq:thm:key:01}
|\uparrow^s_x,\dot{\gamma}^{+}(0)|=|s\gamma(\frac{\pi}{2})|.
\end{equation}

Let $\gamma:[0,\pi]\rightarrow\partial X$ be a closed geodesic with $\gamma(0)=\gamma(\pi)=x$ in $\partial X$, if $|\dot{\gamma}(0),\uparrow_{\gamma_{(0)}}^s|<\frac{\pi}{2}$, then by \autoref{prop:key:rigidity},  the only point on $\gamma$ with distance $\frac{\pi}{2}$ to $s$ is $x=\gamma(0)$.

{\bf Claim 2:}
Let $\gamma:[0,\pi]\rightarrow\partial X$ be a closed geodesic with $\gamma(0)=\gamma(\pi)=x$ in $\partial X$. If $|\dot{\gamma}^{+}(0),\uparrow_{\gamma_{(0)}}^s|=\frac{\pi}{2}$, then $\gamma$ entirely lies in $E$.
\begin{proof}[Proof of Claim 3]
If not, we come back to the case in Claim 1 and Claim 2, it follows that
$$|\uparrow^s_x,\dot{\gamma}^{+}(0)|=|s\gamma(\frac{\pi}{2})|<\frac{\pi}{2},$$
a contradiction.
\end{proof} 
By \cite{GP2018}, we know $\Sigma_{x}(X)=L^{n-1}(\theta)=\SS^{n-3}*[0, \theta]$ for some $\theta\in (0, \pi]$. Any closed geodesic based at $x\in E$ of length $\pi$ satisfies
$$
\dot{\gamma}^{+}(0), \dot{\gamma}^{-}(\pi)\in \partial \Sigma_{x}(X).
$$

{\bf Claim 3:} $\uparrow_{x}^{s}$ is unique, i.e. $\Uparrow_{x}^{s}=\{\uparrow_{x}^{s}\}$.
\begin{proof}[Proof of Claim 3]
Fix any $\uparrow_{x}^{s}\in \Uparrow_{x}^{s}$. Then any geodesic $\gamma([0, \pi])$ with $\dot{\gamma}^{+}(0)=\uparrow_{x}^{s}$ is a closed geodesic with period $\pi$ and lies entirely in $\partial X$. By \eqref{eq:thm:key:01}, we know
$$
|\uparrow^s_x,\dot{\gamma}^{+}(0)|=|s\gamma(\frac{\pi}{2})|=|\uparrow^s_x,\dot{\gamma}^{-}(0)|.
$$
By choosing sufficiently many initial direction of closed geodesics based at $x$, one gets $\uparrow^s_x$ has to have equal distance to any antipodal points in $\partial \Sigma_{x}(X)=\partial L^{n-1}_{\theta}=\SS^{2}$. Therefore, $\uparrow^s_x$ has to be the unique soul of $\Sigma_{x}(X)$.
\end{proof}
By Claim 2, Claim 3 above, one get $\gamma[0, \pi]\subset E$ if and only if $\dot{\gamma}^{+}(0)\in \SS^{n-3}$ whenever $\theta<\pi$ or $\gamma$ always lie in $E$ if $\theta=\pi$. In either case, we have $\dim (E)\ge n-2$.  Since $E\subset \partial X$ is closed and convex, hence in its interior $E$ is locally a submanifold. Therefore, it remains to show $E$ has no boundary. This follows easily from Claim 1. In fact let $x, y \in E$ and $x\ne y$, if the closed geodesic determined by $x, y$ does not lie entirely in $E$, denoted by $\gamma[0, \pi]$, with $\gamma(0)=x, \gamma(\delta)=y$. Then by Claim 1, there exists a unique $\gamma(\pi/2)$ is the unique point closest to the soul $s$. Play the same game with base point $y$, one get a contradiction. Therefore if $E$ has boundary, which is necessarily convex, take $y$ to be the boundary point and $x$ be a point in $E$ that near $y$, such that the closed geodesic determined by $x$ and $y$ is normal to the boundary of $E$. The closed geodesic must leaves $E$, which contradict to we just proved.
\end{proof}

Now we are ready to prove \autoref{thm:main} Case (3). 
\begin{proof}
If $\partial X=\CP^{m}, \HP^{m}, \CaP^{2}$, then clearly it admits no totally geodesic submanifold with codimension $1$, except the case when $\partial X=\CP^{1}$. However in this case if $E$ is the equator of $\CP^{1}$, then it leaves no room for the set $A$, since the diameter of $\CP^{1}$ is $\pi/2$. Therefore by the dimension estimate in \autoref{thm:key:dimension}, we know $E=\partial X$ for all the spaces listed above. Then we can apply \autoref{lem:cone} to get the desired rigidity $X=\Cone(\partial X)$. 

Now we are left with the case $\partial X=\RP^{n-1}$. In this case we have either $E=\partial X$, then by the same reason, $X=\Cone(\partial X)$. Or we have $E=\RP^{n-2}\subset \partial X$ and therefore $A=\{p\}$ the dual set of $E$ inside $\RP^{n-1}$. By \cite{DGGM2018}, we can take the double cover of $X$, branched over the soul point $s$, therefore we get the great sphere $\SS^{n-2}$ as the corresponding set of maximum distance to the soul, and the set of foot points of the soul point are two points in the antipodal position on the boundary $\SS^{n-1}$. Denote by $\theta$ the distance $|s,p|$. Then by \autoref{prop:key:rigidity} and Grove-Petersen's result \autoref{thm:main}(1), we recover the Alexandrov lens structure $L^{n}_{2\theta}/\ZZ_{2}$. This finishes the proof.
\end{proof}

First, we note that $\partial X$ is a space form of constant curvature $1$, therefore it the unique manifold which realized the maximal volume among all Alexandrov metrics with lower curvature bound $1$, then the proof goes verbatim as Proposition 1.1 in \cite{GP2018}. The authors suspect that the same conclusion holds for CROSSes, but we have no proof so far.
\begin{prop}
Let $X\in \alex^n(1)$ with boundary $\partial X$ isometric to the a spherical space form other than $\RP^{n-1}$, then
$$
\rad_s(X)=\frac{\pi}{2}.
$$
\end{prop}

This proposition allows us to remove the radius assumption in the main theorem. Since the space forms other than $\RP^{n-1}$ has no closed proper hypersurface of co-dimension $1$, we know the set of maximum distance to the soul coincide with the set of foot points, this finishes the proof of \autoref{thm:spaceform}. 
	
\bibliographystyle{alpha}
\bibliography{mybib}

\begin{thebibliography}{DGGGM18}

\bibitem[DGGGM18]{DGGM2018}
Qintao Deng, Fernando Galaz-Garc\'{i}a, Luis Guijarro, and Michael Munn.
\newblock Three-dimensional {A}lexandrov spaces with positive or nonnegative
  {R}icci curvature.
\newblock {\em Potential Anal.}, 48(2):223--238, 2018.

\bibitem[GG87]{GG1987}
Detlef Gromoll and Karsten Grove.
\newblock A generalization of {B}erger's rigidity theorem for positively curved
  manifolds.
\newblock {\em Ann. Sci. \'{E}cole Norm. Sup. (4)}, 20(2):227--239, 1987.

\bibitem[GP93]{GP1993}
Karsten Grove and Peter Petersen.
\newblock A radius sphere theorem.
\newblock {\em Invent. Math.}, 112(3):577--583, 1993.

\bibitem[GP18]{GP2018}
Karsten Grove and Peter Petersen.
\newblock A lens rigidity theorem in alexandrov geometry, arxiv:1805.10221
  [math.dg].
\newblock Preprint, 2018.

\bibitem[Per91]{Per1991}
Grisha Perelman.
\newblock {A. D.} {A}lexandrov's spaces with curvatures bounded from below.
  {II}.
\newblock Preprint, 1991.

\bibitem[PP93]{PP1993}
G.~Ya. Perelman and A.~M. Petrunin.
\newblock Extremal subsets in {A}leksandrov spaces and the generalized
  {L}iberman theorem.
\newblock {\em Algebra i Analiz}, 5(1):242--256, 1993.

\bibitem[RW16]{RW2016}
Xiaochun Rong and Yusheng Wang.
\newblock Finite quotient of join in alexandrov geometry, arxiv:1609.07747
  [math.mg].
\newblock Preprint, 2016.

\bibitem[Wil01]{Wil2001}
Burkhard Wilking.
\newblock Index parity of closed geodesics and rigidity of {H}opf fibrations.
\newblock {\em Invent. Math.}, 144(2):281--295, 2001.

\end{thebibliography}

\end{document}